\documentclass[final, hidelinks, onefignum, onetabnum]{siamart220329}

\ifpdf
	\hypersetup{
		pdftitle={The Lanczos Tau Framework for Time-Delay Systems},
		pdfauthor={Evert Provoost and Wim Michiels}
	}
\fi

\usepackage{graphicx}
\graphicspath{{figures}}
\usepackage{booktabs}
\usepackage{physics}
\usepackage{mathtools}
\usepackage{amssymb}
\usepackage{lmodern}
\usepackage[british]{babel}
\usepackage[
	style=authoryear,
	natbib,
	maxbibnames=10,
	maxcitenames=2,
	uniquelist=false,
]{biblatex}
\newcommand{\Htwo}{{H^2}}
\newcommand{\RR}{{\mathbb{R}}}

\newcommand{\CC}{{\mathbb{C}}}
\newcommand{\PP}{{\mathbb{P}}}
\newcommand{\opI}{{\mathcal{I}}}
\newcommand{\opA}{{\mathcal{A}}}
\newcommand{\opB}{{\mathcal{B}}}
\newcommand{\opC}{{\mathcal{C}}}
\newcommand{\opR}{{\mathcal{R}}}
\newcommand{\opE}{{\varepsilon}}
\newcommand{\opD}{{\mathcal{D}}}
\newcommand{\opT}{{\mathcal{T}}}
\newcommand{\opP}{{\mathcal{P}}}

\newcommand{\opS}{{\mathcal{S}}}
\newcommand{\opM}{{\mathcal{M}}}
\newcommand{\opK}{{\mathcal{K}}}
\newcommand{\inp}[2]{{\left\langle#1{,}\,#2\right\rangle}}

\newsiamthm{assumption}{Assumption}

\title{The Lanczos Tau Framework for Time-Delay Systems:\\
	Padé Approximation and Collocation Revisited%
	\thanks{Preprint version.%
		\funding{This work was supported by KU~Leuven project C14/22/092 and
			by FWO-Flanders under grant number G.092.721N.}}}

      \author{Evert Provoost\thanks{KU~Leuven, Department of Computer Science,
          NUMA Research Unit, B-3001 Leuven, Belgium
          (\email{evert.provoost@kuleuven.be},
          \email{wim.michiels@kuleuven.be}).} \and Wim Michiels\footnotemark[2]}

\headers{The Lanczos Tau Framework for Time-Delay
	Systems}{Evert Provoost and Wim Michiels}

\begin{document}

\maketitle

\begin{abstract}
	We reformulate the Lanczos tau method for the discretization of
	time-delay systems in terms of a pencil of operators, allowing for
	new insights into this approach.  As a first main result, we show
	that, for the choice of a shifted Legendre basis, this method is
	equivalent to Padé approximation in the frequency domain.  We
	illustrate that Lanczos tau methods straightforwardly give rise to
	sparse, self-nesting discretizations.  Equivalence is also
	demonstrated with pseudospectral collocation, where the non-zero
	collocation points are chosen as the zeroes of orthogonal
	polynomials.  The importance of such a choice manifests itself in
	the approximation of the $H^2$-norm, where, under mild conditions,
	super-geometric convergence is observed and, for a special case,
	super convergence is proved; both of which are significantly faster
	than the algebraic convergence reported in previous work.
\end{abstract}



\section{Introduction}\label{sec:intro}
We consider a time-delay system described by
\begin{equation}\label{eq:rdde}
	\begin{aligned}
		\dot{\vb{x}}(t) & = A_0 \vb{x}(t)+ A_1 \vb{x}(t-\tau) + B\vb{u}(t), \\
		\vb{y}(t)       & = C\vb{x}(t),
	\end{aligned}
\end{equation}
where $\tau \in \RR_+$ is the constant delay, $\vb{x}(t) \in \CC^n$
the state variable, $\vb{u}(t) \in \CC^{p}$ the input, and
$\vb{y}(t) \in \CC^q$ the output at time $t$. The transfer function of
this system is given by
\begin{equation}\label{eq:tf}
	G(s)=C\pqty{sI_n - A_0 - A_1 e^{-s\tau}}^{-1} B.
\end{equation}

Due to the presence of time-delay, the information required to define
a forward solution at $t=0$, for a given input, is not determined by
$\vb{x}(0)$, but by the function segment
$[-\tau, 0] \ni \theta \mapsto \vb{x}(\theta)$. More generally, the
solution for all $t \geq t_0$ is uniquely defined by the solution for
the time frame $[t_0 - \tau, t_0]$. Hence, the state at time $t$, in
the natural meaning of minimal information to determine the future
evolution, corresponds to the function segment
$\xi_t : [-\tau, 0] \to \CC^n$, with
$\xi_t(\theta) = \vb{x}(t + \theta)$, which explains why a time-delay
model represents an infinite-dimensional dynamical system.  This
infinite-dimensional nature implies that existing techniques for the
analysis and design of delay-free systems cannot be readily
applied. New methods thus have to be developed, most of which start by
discretizing the infinite-dimensional system into a finite-dimensional
approximation.

A common approach in the frequency domain is to replace the exponential function
by a rational approximation, such as the Padé approximant \citep[see
	e.g.][]{GloverMoC1991}. One can also discretize at the level of the state
space. There are two main variants taking this approach. As the system is linear
and time-invariant, one can approximate the solution operator $\opS_T$, which
maps the function $\xi_t$ to $\xi_{t + T}$. Such an approach is for instance
adopted for linearized stability analysis in the bifurcation analysis package by
\citet{EngelborghsNba2002}, where $\opS_T$ is approximated using a linear
multistep method. Similarly, in the context of stability analysis of periodic
delay-differential equations, \citet{ButcherSol2004} propose to discretize
$\opS_T$ as its action on a Chebyshev series approximation of $\xi_t$, where $T$
is taken to be the period.  The other option to discretize at the state space
level, is to look at the infinitesimal generator $\opA$ of the $C_0$-semigroup
$\{\opS_t\}_{t\ge 0}$, with action
\[
	\mathcal{A} \xi_t = \lim_{T \to 0^+} \tfrac{1}{T} (\opS_T - I) \xi_t.
\]
When the input and output in~\eqref{eq:rdde} are taken into account, this
results in a standard state space description of a delay-free system, which
captures part of the system behaviour of the original system.  The main
advantage of these methods is thus that one can often readily apply existing
techniques for the analysis and design of finite-dimensional systems to this
approximation.  The earliest method in this category, to the best of our
knowledge, is the Lanczos tau method of \citet{ItoLTA1986}, which relies on a
truncated Legendre basis. This approach was extended to other bases, and shown
to perform well when computing the characteristic roots, in
\citet{VyasarayaniSaf2014}.

Another, particularly successful, approach in this class is pseudospectral
collocation, introduced by \citet{BredaPDM2005}. It was initially presented for
the eigenvalue problem, but later successfully extended to construct delay-free
approximations, which can for instance be used for bifurcation analysis
\citep{BredaPDo2016}. This method collocates the action of the infinitesimal
generator on a grid of nodes. In practice these are usually the Chebyshev
extremal nodes, a set of nodes which are distributed more densely near the end
points.  This distribution evades Runge's phenomenon, which one can face in more
naive collocation strategies \citep{BoydCaF2001}.  Pseudospectral collocation is
tightly linked to rational approximation of the exponential, which was for
instance used in the initial paper to prove super-geometric convergence of the
characteristic roots.  Aside from theoretical interest, this link can also be
exploited in practice.  An example is a heuristic by \citet{WuRca2012} to select
a discretization degree such that all the characteristic roots to the right of
the imaginary axis are sufficiently well approximated.

Another application of these discretizations is in the computation of
system norms. An example of particular importance is the $\Htwo$-norm,
which, for an exponentially stable, linear, time-invariant system with
transfer function $G$, is given by
\begin{equation}\label{eq:H2}
	\norm{G}_\Htwo = \pqty{\frac{1}{2\pi} \int_{-\infty}^{+\infty} \norm{G(i\omega)}^2_F \dd{\omega}}^{\frac{1}{2}},
\end{equation}
where $i$ is the imaginary unit and $\norm{A}_F = \sqrt{\tr(AA^*)}$ the
Frobenius norm.  This norm is often used in robust control as a measure of
disturbance rejection and, in the context of model reduction, to quantify the
approximation error at the transfer function level. As it is a global
characteristic, in the sense that it depends on the transfer function's
behaviour along the entire imaginary axis, its computation is rather
challenging. For delay-free systems an efficient method involving the algebraic
Lyapunov equation is well known \citep[Lemma~4.6]{ZhouRao1995}. The natural
extension in the delay setting is the so-called delay Lyapunov equation, a
boundary value problem defining a matrix-valued function. In
\citet{JarlebringCaC2011}, a spectral discretization of this equation is
proposed to compute the $\Htwo$-norm, yielding super-geometric convergence of
this norm, at the cost of $\order{n^6N^3}$ operations, where $N$ is the degree
of the approximation.  Alternatively, \citet{VanbiervlietUsd2011} propose to
instead approximate the system using a pseudospectral discretization and then
compute its $\Htwo$-norm through the standard algebraic Lyapunov equation. This
improves the time complexity to $\order{n^3N^3}$ operations, but reduces the
convergence rate to third order algebraic convergence.

In the final section of this article, we illustrate how using a
Lanczos tau method for the approximation of the system allows us to
recover super-geometric convergence, and sometimes even results in
super convergence, giving us the best of both worlds.  This unexpected
improvement served as the initial motivation to revisit the Lanczos
tau method in this work.

\subsection*{Overview}
After reviewing some preliminaries in \cref{sec:prelim}, we present an
operator pencil formulation of the Lanczos tau framework
(\cref{sec:tau}).  We continue by discussing some properties of these
methods in \cref{sec:props}. In particular, we show how these
naturally lead to sparse, nested discretizations and how they are
deeply connected to other approximations.  We prove equivalence to
pseudospectral collocation, when the non-zero collocation points are
chosen as the zeroes of an orthogonal polynomial, and a surprising
link to Padé approximation, when using shifted Legendre polynomials.
Finally, we conclude by illustrating super-geometric convergence, and
for some cases proving super convergence, of the $\Htwo$-norm in
\cref{sec:H2ex}.

\subsection*{Notation}
Throughout this work we will rely on some classical orthogonal
polynomials shifted to the interval $[-\tau, 0]$. To lighten notation
we shall denote these shifted polynomials by their usual names in
the literature for the interval $[-1, 1]$.  In particular we will use
$T_k$ and $U_k$ to denote the shifted Chebyshev polynomials of the
first and second kind, respectively, and $P^{(\alpha, \beta)}_k$ the
similarly transformed Jacobi polynomials, of which the shifted
Legendre polynomials $P_k$ are a special case.  We give a review of
these polynomials in \cref{sec:orthpol}.

\section{Preliminaries}\label{sec:prelim}
Before presenting the Lanczos tau framework, we review some basic
notions and previous work needed in our later development.

\subsection{The abstract Cauchy problem}\label{sec:acp}
To build towards a discretization of the functional differential
equation~\eqref{eq:rdde}, we detail how one can reformulate it in terms of an
abstract Cauchy problem on an infinite-dimensional vector space, where the
unknown, corresponding to the state, is a function defined over an interval of
length~$\tau$. To handle inputs, we explicitly decouple the current state from
the history (resulting in the so-called `head-tail' representation) as in
\citet{CurtainAit1995}.  More precisely, we consider as state space
\begin{equation}\label{eq:headtailspace}
	X := \CC^{n} \times L^2([-\tau,0]; \CC^n).
\end{equation}

Let $\opA: D(\opA) \to X$ be the differential operator with domain
\[ \textstyle
	D(\opA) = \Bqty{(\vb{z},\zeta) \in X : \zeta \in AC,
		\dv{\theta} \zeta \in L^2, \vb{z} = \zeta(0)}
\]
(where $L^2$ and $AC$ have domain $[-\tau, 0]$ and codomain $\CC^n$)
and action
\[
	\opA\,(\vb{z}, \zeta) = \pqty{\pqty{A_{0}\opE_0 + A_1 \opE_{-\tau}} \zeta, \opD \zeta},
\]
where, for later convenience, we introduce evaluation functionals
$\opE_\theta \zeta = \zeta(\theta)$ and differentiation operator
$\opD \zeta = \dv{\theta} \zeta$.  Next, let operators
$\opB : \CC^p \to X$ and $\opC : X \to \CC^q$ be defined by
\[
	\opB \vb{u} = \pqty{B\vb{u}, \vb{0}} \qq{and} \opC z = C \vb{z},
\]
where $\vb{u} \in \CC^p$ and $z = (\vb{z}, \zeta) \in X$.

We can now rewrite~\eqref{eq:rdde} as the abstract Cauchy problem
\begin{equation}\label{eq:acp}
	\begin{aligned}
		\dot{z}(t) & = \mathcal{A}z(t) + \mathcal{B}\vb{u}(t), \\
		\vb{y}(t)  & = \mathcal{C}z(t),
	\end{aligned}
\end{equation}
where $z(t) = (\vb{z}(t), \zeta_t) \in D(\opA)$.  The relation between
corresponding solutions of~\eqref{eq:rdde} and~\eqref{eq:acp} is then
given by
\[
	\vb{z}(t) = \vb{x}(t) \qq{and} \zeta_t(\theta) = \vb{x}(t+\theta),\quad\forall\theta\in[-\tau, 0].
\]
For a more detailed description of the mapping between
representations, and further detail on the inclusion of input and
output, we refer to \citet{CurtainAit1995}.

\subsection{Pseudospectral collocation}\label{sec:pscrev}
As we will discuss relations between pseudospectral collocation and Lanczos tau
methods, we outline how the system~\eqref{eq:acp}, and thus
also~\eqref{eq:rdde}, can be discretized using the former method. A more
comprehensive treatment is given in \citet{BredaSoL2015}. Given a positive
integer $N$, we consider a mesh $\Omega$ of $N+1$ distinct points in the
interval $[-\tau, 0]$, namely
\begin{equation}\label{eq:pscmesh}
	\Omega = \Bqty{\theta_k : k = 0,\dots,N},
\end{equation}
where
\[
	-\tau\leq\theta_0 < \cdots < \theta_{N-1} < \theta_N = 0.
\]
This allows us to replace the continuous space $X$, defined
in~\eqref{eq:headtailspace}, with the space $X_N$ of discrete
functions defined on the mesh $\Omega$, i.e.\ any tuple
$(\vb{z}, \zeta) \in X$ is approximated by a block vector
$\vb{x}_N \in X_N$, with
\[
	\vb{x}_{N,k} = \zeta(\theta_k), \quad k = 0,\dots,N-1, \qq{and} \vb{x}_{N,N} = \vb{z}.
\]

Let $\opP \vb{x}_N$ be the unique $\CC^n$-valued interpolating
polynomial of degree at most $N$, satisfying
\[
	(\opP \vb{x}_N) (\theta_k) = \vb{x}_{N,k},\quad k=0,\dots,N.
\]
This way we can approximate the operator $\mathcal{A}$ by the
finite-dimensional operator $\mathcal{A}_N : X_N \to X_N$, defined by
\[
	\begin{cases}
		\pqty{\mathcal{A}_N \vb{x}_N}_k = \pqty{\opD \opP \vb{x}_N}(\theta_k), \quad k= 0, \dots, N-1, \\
		\pqty{\mathcal{A}_N \vb{x}_N}_N= \pqty{A_0\opE_0 + A_1\opE_{-\tau}} \pqty{\opP \vb{x}_N}.
	\end{cases}
\]
Note that in doing so, we implicitly enforce the boundary condition of
the `head-tail' representation, namely
$\pqty{\opP \vb{x}_N}(0) = \vb{x}_{N,N} = \vb{z}$, where
$\opP \vb{x}_N$ can be seen as the approximation of $\zeta$.

Using the Lagrange representation of $\mathcal{P} \vb{x}_N$,
\[
	\pqty{\opP \vb{x}_N}(\theta) = \sum_{k=0}^N \vb{x}_{N,k}\,\ell_k(\theta),
\]
where the Lagrange polynomials $\ell_k$ are those real-valued
polynomials of degree $N$ satisfying $\ell_k(\theta_j) = \delta_{jk}$,
with $\delta_{jk}$ the usual Kronecker delta, one can get an explicit
matrix expression
\[
	\mathcal{A}_N= \pmqty{[\underline{\opD}] \\ \vb{a}},
\]
where $[\underline{\opD}]$ consists of the first $N$ block rows of the
$(N+1) \times (N+1)$ differentiation matrix
\[
	{[\opD]}_{jk} = I_n \ell'_k(\theta_j),
\]
and $\vb{a}$ is a block row vector with
\[
	\vb{a}_k = A_0 \ell_k(0) + A_1 \ell_k(-\tau),
\]
where $j = 0,\dots,N$ and $k = 0,\dots,N$.

In the same way we can approximate $\mathcal{B}$ and $\mathcal{C}$ by
\[
	\mathcal{B}_N = \pmqty{\vb{0} & \cdots & \vb{0} & B^T}^T \qq{and} \mathcal{C}_N = \pmqty{\vb{0} & \cdots & \vb{0} & C}.
\]
As such, we arrive at a finite-dimensional approximation of~\eqref{eq:rdde}
\begin{equation}\label{eq:pscsys}
	\begin{aligned}
		\dot{\vb{x}}_N(t) & = \mathcal{A}_N \vb{x}_N(t) + \mathcal{B}_N \vb{u}(t),\ \\
		\vb{y}_N(t)       & = \mathcal{C}_N \vb{x}_N(t).
	\end{aligned}
\end{equation}
We can thus also approximate the transfer function~\eqref{eq:tf} by
\begin{equation}\label{eq:psctf}
	J_N(s)=\mathcal{C}_N\pqty{sI_{n(N+1)} - \mathcal{A}_N}^{-1} \mathcal{B}_N.
\end{equation}
The following result on the structure of this approximation was
proved by \citet{GumussoyApc2010}.
\begin{proposition}\label{prop:pscrat}
	The transfer function~\eqref{eq:psctf} satisfies
	\[
		J_N(s) = C\pqty{sI_n - A_0- A_1 p_N(s, -\tau)}^{-1}B,
	\]
	where the function
	\[
		[-\tau, 0] \ni \theta \mapsto p_N(s, \theta)
	\]
	is the unique polynomial of degree $N$ satisfying
	\[
		\begin{cases}
			p_N(s, 0) = 1, \\
			p_N'(s, \theta_k) = s p_N(s, \theta_k),\quad k = 0,\dots,N-1.
		\end{cases}
	\]
	Furthermore, $p_N(s, \theta)$ is a rational function of $s$ for all
	$\theta$.
\end{proposition}
The effect of approximating~\eqref{eq:rdde} by~\eqref{eq:pscsys} can
thus be interpreted, in the frequency domain, as the effect of
approximating the exponential function $s \mapsto e^{-s\tau}$
in~\eqref{eq:tf} by the rational function $s \mapsto p_N(s, -\tau)$.

A common choice of mesh points in the literature consists of scaled
and shifted Chebyshev extremal points, that is,
\[
	\textstyle \theta_k = -\frac{\tau}{2} \pqty{\cos(\frac{\pi k}{N}) + 1},\quad k=0,\dots,N.
\]
The choice of this mesh is motivated by the resulting fast convergence of the
eigenvalues of $\mathcal{A}_N$ to the corresponding characteristic roots
of~\eqref{eq:rdde}. More specifically, in \citet[Theorem 3.6]{BredaPDM2005} it
is proved that super-geometric accuracy, i.e.\ approximation error
$\order{N^{-N}}$, is obtained using these nodes.

To conclude this section, and to introduce the operator approach of
\cref{sec:tau}, note that we can rewrite $\vb{a}$ in terms of a block
vector expression of the evaluation functionals, namely
\[
	\vb{a} = A_0 [\opE_0] + A_1 [\opE_{-\tau}],
\]
where ${[\opE_\theta]}_k = I_n \ell_k(\theta)$. Similarly, we have for
$\mathcal{C}_N$
\[
	C \pmqty{\vb{0} & \cdots & \vb{0} & I_n} = C [\opE_0].
\]

\subsection{Orthogonal polynomials}\label{sec:orthpol}
The Lanczos tau framework presented in the next section will rely on the notion
of a degree-graded series of polynomials orthogonal with respect to an inner
product $\inp{\cdot}{\cdot}$, with induced norm
$\norm{\phi} = \sqrt{\inp{\phi}{\phi}}$.\footnote{We assume the convention that
	$\inp{\cdot}{\cdot}$ is linear in the first argument and antilinear in the
	second.} That is, a set of polynomials $\{\phi_k\}_{k=0}^\infty$ which has as
defining property that $\phi_k$ is of degree $k$ and $\inp{\phi_j}{\phi_k} = 0$
if and only if $j\neq k$.  Usually, the inner product chosen is of the form
\[
	\inp{\phi_j}{\phi_k} = \int_{-\tau}^0  \phi_j(\theta) \overline{\phi_k(\theta)} \, w(\theta) \dd{\theta},
\]
with $w(\theta) \ge 0$, $\forall \theta \in [-\tau, 0]$, the weight
function.  The choice of $w$, together with a normalization condition,
then uniquely defines the orthogonal sequence.

Throughout this work we will use the Jacobi polynomials
$P^{(\alpha, \beta)}_k$ shifted to the interval $[-\tau, 0]$, which
are given by the weight function
\[ \textstyle
	w(\theta) = \pqty{-\frac{2}{\tau}\theta}^\alpha \pqty{\frac{2}{\tau}\theta + 2}^\beta,
\]
and normalization condition
$P^{(\alpha, \beta)}_k(0) = \binom{k+\alpha}{k}$.  As special cases we
have
\[ \textstyle
	T_k = \binom{k - \frac{1}{2}}{k}^{-1} P_k^{\pqty{-\frac{1}{2}, -\frac{1}{2}}} \qq{and}
	U_k = (k + 1) \binom{k + \frac{1}{2}}{k}^{-1} P_k^{\pqty{\frac{1}{2}, \frac{1}{2}}},
\]
the shifted Chebyshev
polynomials of, respectively, the first and second kind, and the
shifted Legendre polynomials
\[
	P_k = P_k^{(0,0)}.
\]
A thorough overview of the properties of these and many other
orthogonal polynomials is given by \citet{SzegoOP1939}.

Finally, we note that, in practice, Chebyshev polynomials are
generally preferred for the approximation of a function
$f : [-\tau, 0] \to \CC$ by a truncated series
\[
	f(\theta) \approx \sum_{k=0}^N \frac{\inp{f}{\phi_k}}{\norm{\phi_k}^2} \phi_k(\theta),
\]
as fast convergence in $N$ is guaranteed for sufficiently smooth functions.  In
particular, \citet[Corollary~2]{MastroianniJoo1995} showed that functions with
$m - 1$ absolutely continuous derivatives, and the $m$th derivative of bounded
variation, give $m$th order algebraic decay of the coefficients of this series.
Additionally, \citet[p.~94]{BernsteinSld1912} proved that for a function which
is analytically continuable to an ellipse in the complex plane, this improves to
geometrical decrease $\order{\rho^N}$, with $0 < \rho < 1$ determined by to the
size of the ellipse.  In the limiting case where the function is entire, this
becomes super-geometric decrease. We will use such a truncated series in the
next section.

\section{The Lanczos tau framework}\label{sec:tau}
We start by selecting an inner product $\inp{\cdot}{\cdot}$ on the space $\PP$
of polynomials $[-\tau, 0] \to \CC$. Let $\Bqty{\phi_k}_{k=0}^\infty$ be a
degree-graded sequence of orthogonal polynomials with respect to this inner
product, as in the previous section. Obviously, for any $N$,
$\Phi_N = \Bqty{\phi_k}_{k=0}^N$ is an orthogonal basis for $\PP_N \subset \PP$,
the space of polynomials of degree at most $N$.  Rather than replacing $X$ by
the discrete space $X_N$, as in \cref{sec:pscrev}, in order to arrive at an
approximation of~\eqref{eq:acp}, and thus~\eqref{eq:rdde}, we will now replace
it by the space $\PP_N^n$, the space of polynomials of degree at most $N$ that
map to $\CC^n$.  The operation of differentiation, encapsulated in the action of
the operator $\mathcal{A}$ in~\eqref{eq:acp}, reduces the degree of a polynomial
with one.  On the left hand side we will thus also have to map an element of
$\PP_N^n$ to an element of $\PP_{N-1}^n$. The idea of \citet{LanczosTIo1938} was
to do so by truncating the series expansion, which was later applied to
functional differential equations by \citet{ItoLTA1986}; it is their method
which we will reformulate as an operator pencil.  For an orthogonal sequence
this truncation namely corresponds to the component-wise orthogonal projector
$\opT_{N-1}$, with action
\[ \textstyle
	\pqty{\opT_{N-1} \xi}_j = {(\xi)}_j - \frac{\inp{{(\xi)}_j}{\phi_N}}{\norm{\phi_N}^2} \phi_N,\quad j=1,\dots,n.
\]

Then letting, as in \cref{sec:acp}, $\opE_\theta$ denote the
evaluation functional in $\theta$, $\opD$ the component-wise
differentiation operator, and $\vb{0}$ the zero polynomial, we propose
the following approximation of~\eqref{eq:rdde}:
\begin{equation}\label{eq:tauop}
	\begin{aligned}
		\pmqty{\opE_0                      \\ \opT_{N-1}}  {\textstyle\dv{t}} \xi_{tN} &= \pmqty{
		A_0 \opE_0 + A_1 \opE_{-\tau}      \\
		\opD} \xi_{tN} + \pmqty{B          \\ \vb{0}} \vb{u}(t), \\
		\vb{y}_N(t) & = C \opE_0 \xi_{tN},
	\end{aligned}
\end{equation}
where $\xi_{tN} \in \PP_N^n$.  The relation between solutions
of~\eqref{eq:tauop} and solutions of~\eqref{eq:rdde}
and~\eqref{eq:acp} can then be described by
\[
	\vb{x}(t) = \vb{z}(t) \approx  \xi_{tN}(0) \qq{and} \zeta_t \approx  \xi_{tN}.
\]

Note that the evolution equation~\eqref{eq:tauop} is in an implicit
form. To show that solutions of the corresponding initial value
problem exist and are uniquely defined, we derive a matrix-vector
representation, induced by expressing elements of $\PP_N^n$ in the basis
$\Phi_N$. That is, $\xi_{tN} = \sum_{k=0}^N \vb{x}_{N,k}(t) \phi_k$,
with $\vb{x}_{N,k}(t) \in \mathbb{C}^n$, $k=0,\dots,N$. Doing so, we
get as block matrix realization of the operators
\[
	\textstyle
	{[\opE_\theta]}_k = I_n \phi_k(\theta), \quad
	{[\opD]}_{jk} = I_n \frac{\inp{\phi'_k}{\phi_j}}{\norm{\phi_j}^2}, \qq{and}
	\textstyle
	{[\opT_{N-1}]}_{jk} = I_n \frac{\inp{\phi_k}{\phi_j}}{\norm{\phi_j}^2} = I_n \delta_{jk},
\]
where $[\,\cdot\,]$ signifies the expression in coordinates,
$j = 0,\dots,N-1$, and $k = 0,\dots,N$. This then gives the explicit
state space realization
\begin{equation}\label{eq:taumatrix}
	\begin{aligned}
		\mathcal{E}_N \dot{\vb{x}}_N(t) & = \mathcal{A}_N \vb{x}_N(t) + \mathcal{B}_N \vb{u}(t), \\
		\vb{y}_N(t)                     & = \mathcal{C}_N \vb{x}_N(t),
	\end{aligned}
\end{equation}
with
\begin{gather*}
	\mathcal{E}_N                 = \pmqty{[\opE_0]                          \\ [\opT_{N-1}]} = \pmqty{
		\phi_0(0)\ \cdots\ \phi_{N-1}(0) & \phi_N(0)                                  \\
		I_{N-1}                        & \vb{0}} \otimes I_n, \\
	\mathcal{A}_N                = \pmqty{A_0[\opE_0] + A_1[\opE_{-\tau}]   \\ [\opD]}, \quad
	\mathcal{B}_N                   = \pmqty{B                                 \\ \vb{0}}, \qq{and}
	\mathcal{C}_N                   = C[\opE_0].
\end{gather*}
This matrix realization is also amenable to implementation. However, note that
the elements of $[\opD]$ should generally not be computed explicitly; see the
book by \citet{BoydCaF2001} for better approaches.

Due to a basic property of orthogonal polynomials, the zeroes of
$\phi_k$ are located in the open interval
$(-\tau, 0)$ \citep[Theorem~3.3.1]{SzegoOP1939}, so it holds that
$\phi_N(0)\neq 0$. Hence, the matrix $\mathcal{E}_N$ is always
invertible.  The invertibility of its matrix expression also implies
that
\[
	\pmqty{\opE_0 \\ \opT_{N-1}}^{-1}
\]
is a well defined operator from $\CC^n \times \PP_{N-1}^n$ to
$\PP_N^n$, and forward solutions of~\eqref{eq:tauop} are thus uniquely
defined.

Taking the Laplace transform of~\eqref{eq:tauop} gives
\begin{equation}\label{eq:taulaplace}
	\begin{aligned}
		s\pmqty{\opE_0                                 \\ \opT_{N-1}} \hat{\xi}_{sN} &= \pmqty{A_0\opE_0 + A_1\opE_{-\tau} \\ \opD} \hat{\xi}_{sN} + \pmqty{B \\ \vb{0}} \hat{\vb{u}}(s), \\
		\hat{\vb{y}}_N(s) & = C \opE_0 \hat{\xi}_{sN},
	\end{aligned}
\end{equation}
where $\hat{f}(s)$ is the transform of $f(t)$.  In this way, we arrive
at an expression of the transfer function of~\cref{eq:tauop}
\begin{equation}\label{eq:tautf}
	G_N(s)=  C \opE_0 \pmqty{s \opE_0 - A_0 \opE_0 - A_1 \opE_{-\tau} \\  s \opT_{N-1} - \opD}^{-1}  \pmqty{B \\ \vb{0}}.
\end{equation}
To conclude this section, we provide a counterpart to \cref{prop:pscrat}.
\begin{proposition}\label{prop:taurat}
	The transfer function~\eqref{eq:tautf} satisfies
	\[
		G_N(s) = C\pqty{s I_n - A_0 - A_1 r_N(s, -\tau)}^{-1}B,
	\]
	where the function
	\[
		[-\tau, 0] \ni \theta \mapsto r_N(s, \theta)
	\]
	is the unique polynomial of degree $N$ satisfying
	\[
		\begin{cases}
			r_N(s, 0) = 1, \\
			\opD r_N(s, \,\cdot\,) = s\opT_{N-1} r_N(s, \,\cdot\,).
		\end{cases}
	\]
	Furthermore, $r_N(s, \theta)$ is a rational function of $s$ for all
	$\theta$.
\end{proposition}

\begin{proof}
	By expanding $r_N(s, \,\cdot\,)$ in a basis, it can easily be seen
	that it is uniquely defined. The second row of the top expression
	in~\eqref{eq:taulaplace} corresponds to a set of homogeneous
	equations, each of which has $r_N(s, \,\cdot\,)$ as a solution, by
	the latter's definition. As a consequence, the solutions of this set
	are of the form $\hat{\xi}_{sN} = r_N(s, \,\cdot\,) \vb{z}(s)$, with
	$\vb{z}(s) \in \CC^n$.  Substituting this form in the first row of
	the top equation leads us to
	\[
		s\vb{z}(s) = A_0 \vb{z}(s) + A_1 r_N(s, -\tau) \vb{z}(s) + B\hat{\vb{u}}(s),
	\]
	while the output equation becomes $\hat{\vb{y}}_N(s) = C
		\vb{z}(s)$.  The assertions follow from solving for $\vb{z}(s)$.
\end{proof}

From the definition of $r_N$ we get the compact representation
\begin{equation}\label{eq:taurnop}
	r_N(s, \theta) = \opE_\theta \pmqty{\opE_0 \\ s\opT_{N-1} - \opD}^{-1} \pmqty{1 \\ \vb{0}}.
\end{equation}
Using this representation we can derive an explicit rational form, as
stated in the following result.
\begin{proposition}\label{prop:ratexp}
	The rational function~\eqref{eq:taurnop} is given by the explicit expression
	\[
		r_N(s, \theta) = \frac{\sum_{k=0}^N \phi_N^{(N-k)}(\theta) \, s^k}{\sum_{k=0}^N \phi_N^{(N-k)}(0) \, s^k},
	\]
	where $\phi^{(k)}$ is the $k$th derivative of $\phi$.
\end{proposition}
\begin{proof}
	By expressing~\eqref{eq:taurnop} in the derivative basis
	$\{\phi_N^{(N-k)}\}^N_{k=0}$, we obtain the companion matrix
	representation
	\[
		r_N(s, \theta) = \pmqty{\phi_N^{(N)}(\theta) \\ \phi_N^{(N-1)}(\theta) \\ \vdots \\ \phi_N(\theta)}^T \pmqty{\phi_N^{(N)}(0) & \phi_N^{(N-1)}(0) & \cdots & \phi_N^{(1)}(0) & \phi_N(0) \\
			s &              -1 &           & &  \\
			&           s &              -1 &           &          \\
			&            &               \ddots &        \ddots &          \\
			&                      &                &        s &        -1  \\
		}^{-1} \pmqty{1 \\ 0 \\ \vdots \\ 0},
	\]
	from which the assertion follows directly.
\end{proof}
Note that, generally, the coefficients of this expression grow rapidly
with $N$. For numerical reasons, solving~\eqref{eq:taurnop} or using
the state space realization is usually preferred in implementation.

\section{Properties}\label{sec:props}
We continue by showing several links to previously proposed methods,
which can, partially, be unified under the Lanczos tau operator
framework. Additionally, we present how Lanczos tau methods naturally
lead to nested, sparse discretizations.

\subsection{Relation to pseudospectral collocation}
Since Lanczos tau methods for differential equations are well known to
correspond to collocation in the zeroes of the truncated
polynomial \citep{LanczosTIo1938}, a similar intimate connection
between pseudospectral collocation and the approximation scheme of the
previous section is expected; the following result holds.
\begin{theorem}\label{thm:eq}
	Assume that the non-zero mesh points of $\Omega$, as defined
	in~\eqref{eq:pscmesh}, are chosen as the zeroes of $\phi_N$, that is
	\[
		\phi_N(\theta_k)=0,\quad k = 0, \dots, N-1.
	\]
	Then for any $N$, the finite-dimensional
	approximations~\eqref{eq:pscsys} and~\eqref{eq:tauop} are
	equivalent, i.e.\ $J_N(s) = G_N(s)$.
\end{theorem}
\begin{proof}
	Note that~\eqref{eq:pscsys} can also be derived from the relations
	\[
		\begin{aligned}
			\textstyle
			 & \begin{cases}
				   \opE_0 \dv{t} \eta_{tN} = \pqty{A_0 \opE_0 + A_1 \opE_{-\tau}} \eta_{tN}+ B \vb{u}(t), \\
				   \textstyle
				   \opE_{\theta_k} \dv{t} \eta_{tN} = \opE_{\theta_k}\opD \eta_{tN}, \quad k = 0,\dots, N-1,
			   \end{cases} \\
			 & \quad\vb{y}_N(t) = C \opE_0 \eta_{tN},
		\end{aligned}
	\]
	with $\eta_{tN} \in \PP_N^n$, by expressing $\eta_{tN}$ in the
	Lagrange basis with respect to~\eqref{eq:pscmesh}.  Additionally,
	using the definition of $\opT_{N-1}$, the bottom row of the first
	equation in~\eqref{eq:tauop} reads
	\[ \textstyle
		\pqty{\dv{t} \xi_{tN}}_j - \inp{\pqty{\dv{t} \xi_{tN}}_j}{\phi_N} \frac{\phi_N}{\norm{\phi_N}^2}  =  \pqty{\opD \xi_{tN}}_j, \quad j = 1,\dots,n,
	\]
	which, under the above conditions on $\Omega$, implies
	\[ \textstyle
		\opE_{\theta_k} \dv{t} \xi_{tN} = \opE_{\theta_k} \opD \xi_{tN}, \quad k = 0, \dots, N-1.
	\]
	Hence, the conditions imposed on the evolution of the polynomial $\xi_{tN}$,
	as a function of $t$, imply the conditions imposed on the evolution of the
	polynomial $\eta_{tN}$. But since each set of conditions uniquely defines
	the flow, they must be equivalent.
\end{proof}

This connection allows one to reuse the large number of results that
were developed for collocation-based methods. In particular, we
recover super-geometric convergence of the eigenvalues when using a
Chebyshev basis of the first or second kind, as this corresponds to
using the zeroes of $T_N$ or $U_N$ as the non-zero collocation
points \citep[Theorem~5.1]{BredaSoL2015}. Such a result is not
unexpected as the conditions in \cref{prop:pscrat,prop:taurat}, in the
limit, define the function $\theta \mapsto e^{s \theta}$. These
methods are thus grounded in the approximation of the exponential by a
polynomial, obtained through collocation or the truncation of a
series, respectively.  From the discussion in \cref{sec:orthpol}, and
analogous results for interpolation through Chebyshev nodes,
super-geometric convergence is observed for this function as it is
entire.

\subsection{Sparse, nested discretizations}\label{sec:selfnest}
From a degree-graded orthogonal sequence $\{\phi_k\}_{k=0}^\infty$, we
can trivially form a sequence of nested bases
$\Phi_0 \subset \Phi_1 \subset \dots$ for
$\PP_0 \subset \PP_1 \subset \dots$ respectively. As a consequence, it
is easy to construct a matrix realization of~\eqref{eq:tauop} which is
also nested, in the sense that the resulting matrices for degree $N_1$
are submatrices of those at degree $N_2$, for $N_1 < N_2$. In fact,
the matrices in~\eqref{eq:taumatrix} have this property.

This nesting of state space representations can be exploited in Krylov
algorithms for characteristic roots computation and model reduction of time
delay systems of high dimension, as in the infinite Arnoldi method introduced by
\citet{JarlebringAKM2010}. This leads to far cheaper and far more flexible
methods, as this nesting allows the reuse of previous computations whilst
adaptively changing the discretization degree.

As an example, the self-nesting discretization initially derived for
this purpose in the above article is based on collocation in zero and
the zeroes of $U_N$.  By \cref{thm:eq} we can thus recast this in the
framework of \cref{sec:tau}, as this corresponds, up to a basis
transform, to the choice $\Phi_N = \Bqty{U_k}_{k=0}^N$.

Furthermore, note that when expressing~\eqref{eq:tauop} in a basis, it is not
necessary to use $\Phi_N$ as basis for the input side of the operators (on the
contrary, using $\Phi_N$ to represent the output yields neater representations
of $\opT_{N-1}$ and is thus generally preferred).  One can, for instance, choose
the Chebyshev polynomials of the first kind on the input side and those of the
second kind as $\Phi_N$. This leads to a highly sparse representation of $\opD$,
with $\order{N}$ non-zeroes instead of $\order{N^2}$, as exploited in the
ultraspherical method introduced by \citet{OlverAFa2013}, yielding even more
computational gains.  In fact, the matrix representation of
\citet{JarlebringAKM2010} corresponds to this choice, up to a scaling of the
rows.

\subsection{Relation to Padé approximation}\label{sec:pade}
The Padé approximant of $e^s$ near zero and the Legendre polynomials are well
known to be linked, as shown by \citet{AhmadTop1997}.  In the context of the
approximation of time-delay systems, such a connection has also been
demonstrated by \citet{BajodekFds2021}, which inspires a potential similar link
for the Lanczos tau framework.  Such a connection indeed turns out to exist, as
shown by the following result.

\begin{theorem}\label{thm:pade}
	For the choice $\Phi_N = \Bqty{P_k}_{k=0}^N$, the rational function
	\[
		s \mapsto r_N(s, -\tau),
	\]
	as in \cref{prop:taurat}, is an $(N,N)$ Padé approximant of
	$e^{-\tau s}$ near zero.
\end{theorem}
\begin{proof}
	From \cref{prop:taurat} we know that $r_N(s, -\tau)$ is a rational
	function of (at most) type $(N, N)$. The defining property of a Padé
	approximant of this type, and thus what we must show, is that the
	first $2N + 1$ moments match those of the exponential at zero, i.e.\
	\[  \textstyle
		\bqty{\dv[n]{s} r_N(s, -\tau)}_{s=0} = \bqty{\dv[n]{s} e^{-\tau s}}_{s=0}= {(-\tau)}^n \quad \forall n \le 2N.
	\]

	As the operators involved in~\eqref{eq:taurnop} linearly map between
	finite-dimensional spaces, we can apply the analogue of the
	derivative of an inverse matrix, yielding
	\begin{align*} \textstyle
		\bqty{\dv[n]{s} r_N(s, -\tau)}_{s=0} &
		\textstyle
		= \bqty{-\dv[n - 1]{s}\, \opE_{-\tau} \spmqty{\opE_0                                               \\ s\opT_{N-1} - \opD}^{-1} \spmqty{\vb{0}^* \\ \opT_{N-1}} \spmqty{\opE_0 \\ s\opT_{N-1} - \opD }^{-1} \spmqty{1 \\ \vb{0}}}_{s=0} \\
		                                     & = \bqty{n!\,\pqty{-1}^n\, \opE_{-\tau} \bqty{\spmqty{\opE_0 \\ s\opT_{N-1} - \opD}^{-1} \spmqty{\vb{0}^* \\ \opT_{N-1}}}^n \spmqty{\opE_0 \\ s\opT_{N-1} - \opD }^{-1} \spmqty{1 \\ \vb{0}}}_{s=0} \\
		                                     & = n!\, \opE_{-\tau} \bqty{\spmqty{\opE_0                    \\ \opD}^{-1} \spmqty{\vb{0}^* \\ \opT_{N-1}}}^n \spmqty{\opE_0 \\ -\opD }^{-1} \spmqty{1 \\ \vb{0}},
	\end{align*}
	with $\vb{0}^* f = 0$. Let
	$\opM_N = \spmqty{\opE_0 \\ \opD}^{-1} \spmqty{\vb{0}^* \\
			\opT_{N-1}}$ and note that
	\[
		(\opM_N f)(\theta) = \int_0^\theta \bqty{f(\xi) - \inp{f}{P_N}P_N(\xi)} \dd{\xi}.
	\]
	Additionally, for
	$f_0 = \spmqty{\opE_0 \\ -\opD }^{-1} \spmqty{1 \\ \vb{0}}$ we
	have $f_0(\theta) = 1$, hence
	\[ \textstyle
		\bqty{\dv[n]{s} r_N(s, -\tau)}_{s=0} = n! \, \opE_{-\tau} \opM_N^n f_0.
	\]
	It thus remains to show that for $n \le 2N$,
	\[
		n!\,\opE_{-\tau} \opM_N^n f_0 = {(-\tau)}^n.
	\]

	The operator $\opM_N$ maps from $\PP_N$ to $\PP_N$. If we embed this space
	in the space $\PP$ of polynomials of arbitrary degree, we can split
	$\opM_N = \opI + \opK_N$, where
	$(\opI f)(\theta) = \int_0^\theta f(\xi) \dd{\xi}$ and
	$(\opK_N f)(\theta) = -\int_0^\theta \inp{f}{P_N}P_N(\xi) \dd{\xi}$.  We
	know that
	$\frac{2}{\tau}(2k+1)P_k(\theta) = \dv{\theta} \pqty{P_{k+1}(\theta) -
			P_{k-1}(\theta)}$ for $k \ge 1$ \citep[consequence of][eq.\
		4.7.29]{SzegoOP1939}, hence, by the fundamental theorem of calculus and the
	fact that $P_{k+1}(0) = P_{k-1}(0) = 1$,
	\[
		\opK_N f = -\frac{\inp{f}{P_N}}{\frac{2}{\tau}(2k+1)}\pqty{P_{N + 1} - P_{N - 1}}.
	\]
	As $P_N$ is orthogonal to all polynomials of degree less than $N$,
	we have
	\begin{equation}\label{eq:pdprf1}
		\opK_N g = \vb{0},\quad \forall g \in \PP_{N-1}.
	\end{equation}

	Note that ${(\opI + \opK_N)}^n$ is the sum of all possible
	sequences of $\opI$ and $\opK_N$ of length $n$.  As a consequence
	of~\eqref{eq:pdprf1}, we know that any subsequence of operators at
	most $N$ long and containing $\opK_N$, gives the zero function
	when applied to $f_0$. Thus ${(\opI + \opK_N)}^n f_0$ consists of
	$(\opI^n f_0)(\theta) = \frac{\theta^n}{n!}$ and terms of the form
	$\opI^m \opK_N f$, with $m < n - N$.  Applying the same
	integration rule for $P_k$ from before, we get the following forms
	for $\opI^m \opK_N f$.
	\medskip
	\begin{center}
		\begin{tabular}{cc}
			\toprule
			$m$      & expression                                              \\
			\midrule
			0        & $\alpha (P_{N+1} - P_{N-1})$                            \\
			1        & $\beta (P_{N+2} - P_N) + \gamma(P_N - P_{N-2})$         \\
			$\vdots$ & $\vdots$                                                \\
			$N-1$    & $\chi (P_{2N} - P_{2N-2}) + \cdots + \omega(P_2 - P_0)$ \\
			\bottomrule
		\end{tabular}
	\end{center}
	\medskip
	As $P_k(-\tau) = {(-1)}^k$ \citep[eq.\ 4.1.4]{SzegoOP1939}, we get
	\[
		\opE_{-\tau} \opI^m \opK_N f = 0, \quad \forall m < N.
	\]
	Since $n \le 2N$ we have $m < 2N - N = N$, thus the last result holds, hence
	\[
		n! \, \opE_{-\tau} \opM_N^n f_0 = n! \, \opE_{-\tau} \opI^n f_0 = {(-\tau)}^n, \quad \forall n \le 2N,
	\]
	which is precisely what had to be shown.
\end{proof}
This result and \cref{prop:taurat} then readily give the following.
\begin{corollary}\label{cor:padeintf}
	For the choice $\Phi_N = \Bqty{P_k}_{k=0}^N$, the Lanczos tau
	discretization~\eqref{eq:tauop} is a state space representation of
	the transfer function of the original system~\eqref{eq:tf} where
	the exponential $s \mapsto e^{-\tau s}$ is replaced by a Padé
	approximant of type $(N, N)$ around zero.
\end{corollary}
Of course, an analogous result holds for pseudospectral collocation,
via \cref{thm:eq}.

Finally, we can use this link and \cref{prop:ratexp} to give an
explicit expression for the resulting Padé approximant.
\begin{corollary}
	The $(N, N)$ Padé approximant of $s \mapsto e^{-\tau s}$ around zero
	has as rational expression
	\[
		\frac{\sum_{k=0}^N P_N^{(N-k)}(-\tau) \, s^k}{\sum_{k=0}^N P_N^{(N-k)}(0) \, s^k}.
	\]
\end{corollary}
Note that this expression of the Padé approximant is identical to
expression~(3.4) in \citet{AhmadTop1997}. This correspondence can thus
be used to give an alternative proof of \cref{thm:pade}, be it more
indirect than the approach presented here.

Whilst the results preceding this subsection straightforwardly extend to more
general settings such as multiple or distributed delays, those of this
subsection are unique to the single delay case. Indeed, if we replace
$A_0 \opE_0 + A_1 \opE_{-\tau}$ by some other bounded linear functional
in~\eqref{eq:tauop}, we see by the reasoning of \cref{prop:taurat} that
$\theta \mapsto r_N(s, \theta)$ will be evaluated in multiple points, not just
$-\tau$. From \cref{prop:ratexp}, however, we know that the poles of $r_N$ are
independent of $\theta$, hence an analogue of \cref{cor:padeintf} cannot hold,
as the $(N,N)$ Padé approximants of e.g.\ $s \mapsto e^{-s}$ and
$s \mapsto e^{-2s}$ do not have common poles.

\section{An application: computing the $\Htwo$-norm}\label{sec:H2ex}
We conclude this article by illustrating some unexpected improvements when using
a Lanczos tau method for the computation of the $\Htwo$-norm, as defined
by~\eqref{eq:H2}, of~\eqref{eq:rdde}.\footnote{The codes used to produce the
	figures in this section are available at
	\url{https://doi.org/10.5281/zenodo.12088674}.}  For this problem two methods
have been proposed in the past. One approximates the so-called delay Lyapunov
equation \citep{JarlebringCaC2011}; the other uses a state space discretization,
namely pseudospectral collocation in Chebyshev extremal
nodes \citep{VanbiervlietUsd2011}.  Whilst the latter has a lower time
complexity, $\order{n^3N^3}$ instead of $\order{n^6N^3}$, it only achieves third
order algebraic convergence compared to the super-geometric rate of the former.

We will see here that the choice of discretization can have a profound
impact on the convergence of the second method.  In particular, a
Lanczos tau method satisfying the following assumption appears to
recover super-geometric convergence and sometimes even displays super
convergence, i.e.\ it gives the exact result for all $N$ larger than
some finite $N_0$, as we will demonstrate in \cref{prop:superconv}.
\begin{assumption}\label{assum:oddeven}
	The basis function $\phi_k$ is symmetric when $k$ is even and
	antisymmetric when $k$ is odd, i.e.\
	$\phi_k(-\tau - \theta) = {(-1)}^k \phi_k(\theta)$,
	$\forall \theta \in [-\tau, 0]$.
\end{assumption}
Note that when shifted from the interval $[-\tau, 0]$ to the interval
$[-\frac{\tau}{2}, \frac{\tau}{2}]$, symmetric and antisymmetric
functions correspond to the classical notion of even and odd
functions, respectively.

For many of the bases used in this paper, this property is a direct consequence
of the defining weight function being symmetric \citep[eq.\ 2.3.3]{SzegoOP1939}.
The usual pseudospectral collocation method using Chebyshev extremal points does
not satisfy any analogous property, as the non-zero collocation points are not
symmetric. The less common Chebyshev zeroes extended with the node $0$, however,
are, per \cref{thm:eq}, equivalent to a Lanczos tau method that does satisfy
this symmetry condition.

The idea of \citet{VanbiervlietUsd2011}, here translated to the Lanczos tau
method, is simple: approximate the $\Htwo$-norm of the system~\eqref{eq:rdde} by
the $\Htwo$-norm of the approximation~\eqref{eq:tauop}. As is well
known \citep[Lemma~4.6]{ZhouRao1995} we can compute the latter, if the resulting
system is exponentially stable, as
\[
	\norm{G_N}_\Htwo = \sqrt{\tr(\mathcal{C}_NV \mathcal{C}_N^T)},
\]
where the matrix $V \in \CC^{nN\times nN}$ is the solution of the generalized
symmetric Lyapunov equation, in the notation of~\eqref{eq:taumatrix},
\begin{equation}\label{eq:symlyap}
	\mathcal{A}_N V \mathcal{E}_N^T + \mathcal{E}_N V \mathcal{A}_N^T = -\mathcal{B}_N \mathcal{B}_N^T.
\end{equation}

We can reinterpret this matrix equation in terms of operations on
polynomials, which will turn out to be useful in explaining some of
the results in the remainder of this section. To this end, note that
we can identify the solution $V$ by a bivariate polynomial
\[ \textstyle
	U(\theta, \theta') = \sum_{j=0}^N \sum_{k=0}^N V_{jk}\, \phi_j(\theta)\phi_k(\theta') \in \CC^{n \times n},
\]
with $V_{jk}$ the relevant $n \times n$ subblock of $V$.  Formally, as
our choice of coordinates with respect to the basis $\Phi_N$ imposes
a bijection between $\CC^{nN}$ and $\PP_N^n$, we similarly have an
induced bijection between the tensor spaces
$\CC^{nN \times nN} \cong \CC^{nN} \otimes \CC^{nN}$ and
$\PP_N^n \otimes \PP_N^n$ via the basis $\Phi_N \otimes \Phi_N$.

Inspired by this bijection, we extend our operator notation to
bivariate polynomials, to have left multiplication mean application to
the first variable and transposed right multiplication to be
application to the second variable, i.e.\ we have
$\opE_\theta U \opE_{\theta'}^T := U(\theta, \theta')$.  This notation
is justified by the tensor nature of $U$. Indeed, denoting by
$[\,\cdot\,]$ the matrix realization as in~\cref{sec:tau}, we get
\[ \textstyle
	\opE_\theta U \opE_{\theta'}^T = [\opE_\theta] V [\opE_{\theta'}]^T \in \CC^{n \times n}.
\]
We then have the following result.
\begin{proposition}\label{prop:delaylyap}
	The $\Htwo$-norm of the transfer function~\eqref{eq:tautf} of the
	Lanczos tau approximation~\eqref{eq:tauop}, if exponentially stable,
	is given by
	\[
		\norm{G_N}_\Htwo = \sqrt{\tr(C \opE_0 U \opE_0^T C^T)},
	\]
	where the bivariate polynomial
	$U : [-\tau, 0] \times [-\tau, 0] \to \CC^{n\times n}$, of degree
	$(N, N)$, is the unique solution of the system
	\begin{equation}\label{eq:discdellyap}
		\begin{cases}
			\opD U \opE_0^T +  \opT_{N-1} U \pqty{\opE_0^TA_0^T + \opE_{-\tau}^T A_1^T} = \vb{0}, \\
			\opD U \opT^T_{N-1} + \opT_{N-1} U \opD^T = \vb{0},                                   \\
			\opE_\theta U \opE^T_{\theta'} = \pqty{\opE_{\theta'} U \opE^T_{\theta}}^T,           \\
			\pqty{A_0\opE_0 + A_1\opE_{-\tau}} U \opE^T_0 + \opE_0 U \pqty{\opE_0^TA_0^T + \opE_{-\tau}^TA_1^T} = -BB^T.
		\end{cases}
	\end{equation}
\end{proposition}
\begin{proof}
	Note that we can write~\eqref{eq:symlyap}, without loss of
	generality, as
	\[
		\spmqty{A_0 [\opE_0] + A_1 [\opE_{-\tau}] \\ [\opD]} V \spmqty{[\opE_0] \\ [\opT_{N-1}]}^T +  \spmqty{[\opE_0] \\ [\opT_{N-1}]} V \spmqty{A_0 [\opE_0] + A_1 [\opE_{-\tau}] \\ [\opD]}^T = -\spmqty{B \\ \vb{0}}\spmqty{B \\ \vb{0}}^T.
	\]
	Similarly, we have
	$\norm{G_N}^2_\Htwo = \tr(C [\opE_0] V [\opE_0]^T C^T)$. The
	assertions can now be recovered using the earlier identification
	between matrices and bivariate polynomials, by splitting these
	expressions into their sub\-blocks and noting that $V = V^T$
	implies
	$\opE_\theta U \opE^T_{\theta'} = \pqty{\opE_{\theta'} U
		\opE^T_{\theta}}^T$.  Finally, uniqueness and existence are
	direct consequences of the corresponding properties
	of~\eqref{eq:symlyap}.
\end{proof}

We can use this result to connect the approach adopted here with the method of
\citet{JarlebringCaC2011}.  The latter is based on the following
characterization of the $\Htwo$-norm, again assuming the system~\eqref{eq:rdde}
is exponentially stable:
\[
	\norm{G}_\Htwo = \sqrt{\tr(C\lambda(0)C^T)},
\]
where $\lambda \in C^1([-\tau, \tau]; \CC^{n\times n})$ is the
solution of the so-called delay Lyapunov equation
\begin{equation}\label{eq:delaylyap}
	\begin{cases}
		\lambda'(t) + \lambda(t)A_0^T + \lambda(t + \tau)A_1^T = \vb{0}, \quad t \in [-\tau, 0), \\
		\lambda(-t) = {\lambda(t)}^T,                                                            \\
		A_0\lambda(0) + A_1\lambda(-\tau) + \lambda(0)A_0^T + \lambda(\tau)A_1^T = -BB^T.
	\end{cases}
\end{equation}
At least intuitively, we now have a link between both methods via
$U(\theta, \theta') \approx \lambda(\theta - \theta')$.  Indeed,
assuming $\opT_{N-1}$ to go to the identity operator as
$N \to \infty$, the top and bottom equations of~\eqref{eq:discdellyap}
and~\eqref{eq:delaylyap} match.  We recover
$\lambda(-t) = {\lambda(t)}^T$ from
$\opE_\theta U \opE_{\theta'}^T = \pqty{\opE_{\theta'} U
		\opE_\theta^T}^T$ by setting $\theta = t$ and $\theta' = 0$ for
$t \le 0$ and $\theta = 0$ and $\theta' = -t$ for $t > 0$.  Finally, the
remaining equation guarantees that, in the limit, $U(\theta, \theta')$
is a function of $\theta - \theta'$. We can thus interpret the bivariate
polynomial associated with the solution of~\eqref{eq:symlyap} as an
approximation of $(\theta, \theta') \mapsto \lambda(\theta - \theta')$
with respect to $\Phi_N \otimes \Phi_N$.

Finally, as an additional property, these equations show an
interesting symmetry in the scalar case.
\begin{lemma}\label{lem:reversal}
	For a Lanczos tau discretization satisfying \cref{assum:oddeven} of
	a scalar system~\eqref{eq:rdde}, i.e.\ a system with $n=1$, it holds
	that the solution of~\eqref{eq:discdellyap}
	\[
		U = \opR U \opR^T,
	\]
	with $\opR$ the reversal operator, such that
	$\opE_\theta \opR = \opE_{-\theta-\tau}$.
\end{lemma}
\begin{proof}
	Decomposing $U$ in the derivative basis of $\phi_N$, that is,
	writing
	\[ \textstyle
		U = \sum_{j=0}^N \sum_{k=0}^N W_{jk} \, \Psi_{jk},
	\]
	with
	$\Psi_{jk}(\theta, \theta') = \phi_N^{(N-j)}(\theta)\,
		\phi_N^{(N-k)}(\theta')$, we note that
	$\opD U \opT^T_{N-1} = -\opT_{N-1} U \opD^T$ implies
	\[ \textstyle
		\sum_{j=1}^N \sum_{k=0}^{N-1} W_{jk} \, \Psi_{{j-1},k} =
		-\sum_{j=0}^{N-1} \sum_{k=1}^N W_{jk} \, \Psi_{j,{k-1}},
	\]
	and thus $W_{j+1,k} = -W_{j,k+1}$. Hence, the coefficients on an
	anti-diagonal of $W$ have constant magnitude and alternating
	sign. As we additionally have $W_{jk} = W_{kj}$ from
	$\opE_\theta U \opE^T_{\theta'} = \pqty{\opE_{\theta'} U
		\opE^T_{\theta}}^T$, the anti-diagonals of even length must be
	zero, thus $W_{jk} = 0$ if $j+k$ is odd.  Furthermore,
	\cref{assum:oddeven} implies that $\phi^{(N-k)}_N$ is symmetric
	when $k$ is even and antisymmetric when $k$ is odd. Hence, for
	$i + j$ even, the bivariate polynomial $\Psi_{jk}$ is the product
	of either two symmetric or two antisymmetric polynomials and thus
	$\Psi_{jk} = \opR\Psi_{jk}\opR^T$. The assertion then follows from
	the linearity of $\opR$ and the structure of $W$.
\end{proof}
Note that from the point of view of the algebraic Lyapunov
equation~\eqref{eq:symlyap}, the same structure on $W$ follows from
the companion matrix structure of the equation, when expressed in the
above basis, as shown by \citet{BarnettTLm1967}.

\subsection{General $A_0$ and $A_1$}
If we approximate the $\Htwo$-norm using different state space
discretizations, namely pseudospectral collocation in Chebyshev
extremal nodes and a Lanczos tau method satisfying
\cref{assum:oddeven}, we consistently see results similar to
\cref{fig:usual-conv}. We get super-geometric convergence for the
Lanczos tau method instead of the third order convergence of the other
discretization as described in \citet{VanbiervlietUsd2011}.

\begin{figure}
	\centering
	\includegraphics[width=.7\textwidth]{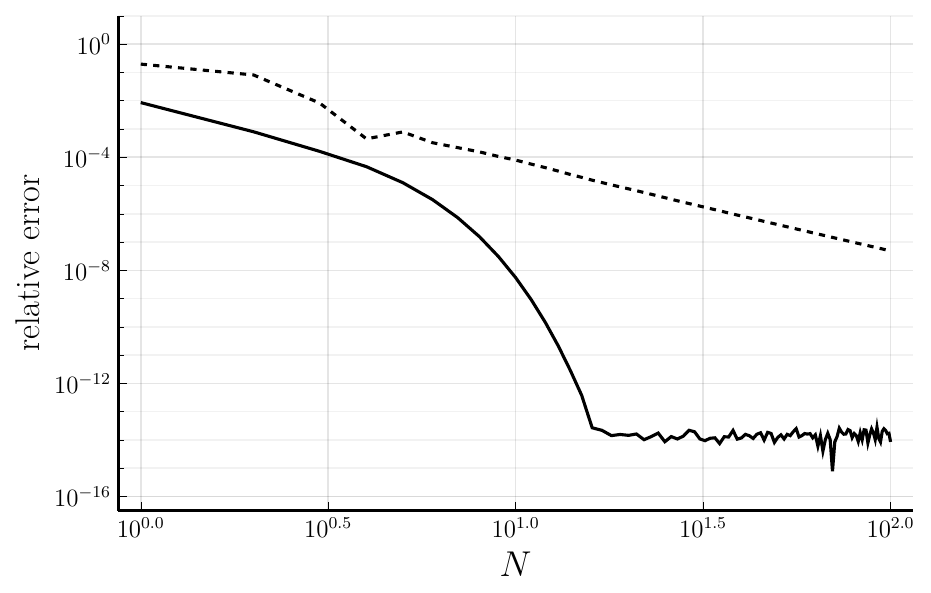}
	\caption{\label{fig:usual-conv} Convergence of the $\Htwo$-norm
		of~\eqref{eq:rdde} with $A_0 = \spmqty{-2 & 1 \\ 3 & -8}$,
		$A_1 = \spmqty{-1 & -1 \\ -1 & -1}$, $\tau = 1$, and
		$B = C = I_2$, for different discretizations. The solid line
		corresponds to the Lanczos tau method with Chebyshev polynomials
		of the second kind, the dashed line to pseudospectral
		collocation in Chebyshev extremal nodes.}
\end{figure}

\begin{figure}
	\centering
	\includegraphics[width=.7\textwidth]{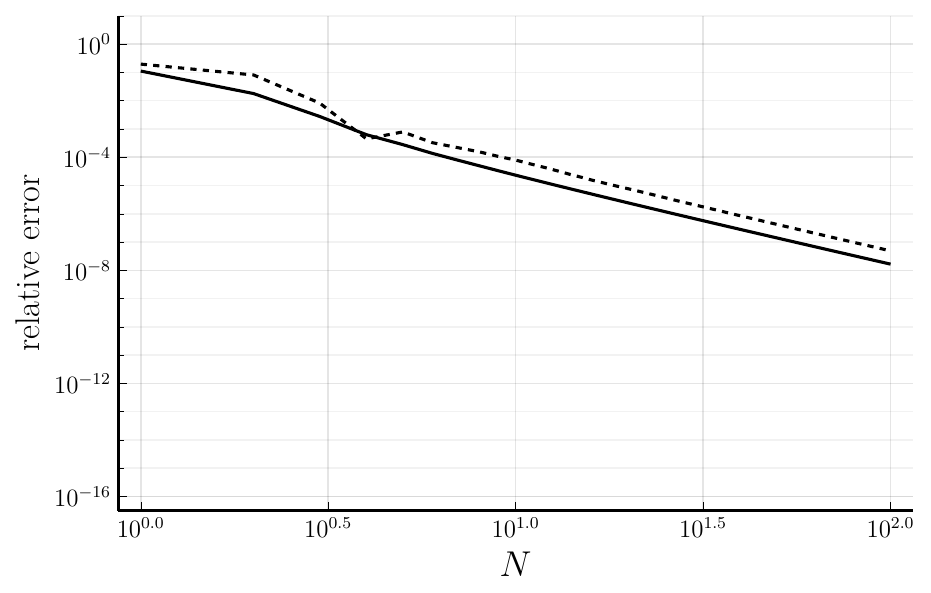}
	\caption{\label{fig:jacobi-conv} Convergence of the $\Htwo$-norm
		of the system of \cref{fig:usual-conv}, for different
		discretizations. The solid line corresponds to the Lanczos tau
		method with Jacobi $\pqty{-\frac{1}{2},-\frac{3}{4}}$
		polynomials, which do not satisfy \cref{assum:oddeven}, the
		dashed line to pseudospectral collocation in Chebyshev extremal
		nodes.}
\end{figure}

Note that the $\Htwo$-norm depends on the behaviour of the transfer
function along the entire imaginary line; however, the rational
approximation of \cref{prop:taurat} is only accurate around zero.
From this perspective, algebraic convergence is the most that one
would expect from such a method \citep{VanbiervlietUsd2011}.  The
observed improvement in convergence is thus highly surprising,
especially as we did not change the computational cost.  Part of the
explanation might stem from the link to the delay Lyapunov equation.
As its solution is an analytic function defined on a finite domain, we
could get super-geometric convergence for a Chebyshev basis if we in
fact implicitly solve this equation using a spectral
method \citep{TadmorTEA1986}.  However, similar convergence would then
be expected of other reasonably behaved bases, such as the Jacobi
polynomials $P_k^{\pqty{-\frac{1}{2},-\frac{3}{4}}}$ which nonetheless
degrade back to third order convergence (see
\cref{fig:jacobi-conv}). As we only note this deterioration when
\cref{assum:oddeven} is not satisfied, we presume that such a symmetry
condition also plays a major part in explaining the observed
super-geometric convergence.  A hint is given by the following result,
which shows that the rational approximation qualitatively matches the
exponential better on the imaginary axis, in the sense that its
magnitude equals one, precisely when this property is satisfied.
\begin{proposition}\label{prop:ratmag}
	If $\Phi_N$ is chosen such that \cref{assum:oddeven} is satisfied,
	we have
	\[
		\abs{r_N(i\omega, -\tau)} = 1, \quad \forall \omega \in \RR,
	\]
	where $r_N$ is as in \cref{prop:taurat}.
\end{proposition}
\begin{proof}
	From \cref{prop:ratexp} we have
	\[
		r_N(i\omega, -\tau) = \frac{\sum_{k=0}^N i^k \phi_N^{(N-k)}(-\tau) \, \omega^k}{\sum_{k=0}^N i^k \phi_N^{(N-k)}(0) \, \omega^k}.
	\]
	\Cref{assum:oddeven} implies
	$\phi_N^{(N-k)}(-\tau) = (-1)^k \phi_N^{(N-k)}(0)$, hence the
	denominator is the complex adjoint of the numerator.
\end{proof}

\subsection{A case with super convergence}
We consider the system~\eqref{eq:rdde} in the scalar case, with
$A_0 = A_1 = a < 0$.  The zero solution of this system is then known to be
exponentially stable for any $\tau \ge 0$, as proved by
\citet[Theorem~1]{HayesRot1950}.  Impressively, we see that the discretizations
based on the Lanczos tau method, again under \cref{assum:oddeven}, already give
the exact result at $N=1$ (see \cref{fig:super-conv}; the slight increase in
error as $N$ increases is due to rounding errors in the underlying matrix
operations). This is once again rather surprising from the perspective of an
integral over an unbounded domain, as the transfer functions barely match, even
near zero (see \cref{fig:tf-approx}). We conclude by proving this effect.

\begin{figure}
	\centering
	\includegraphics[width=.7\textwidth]{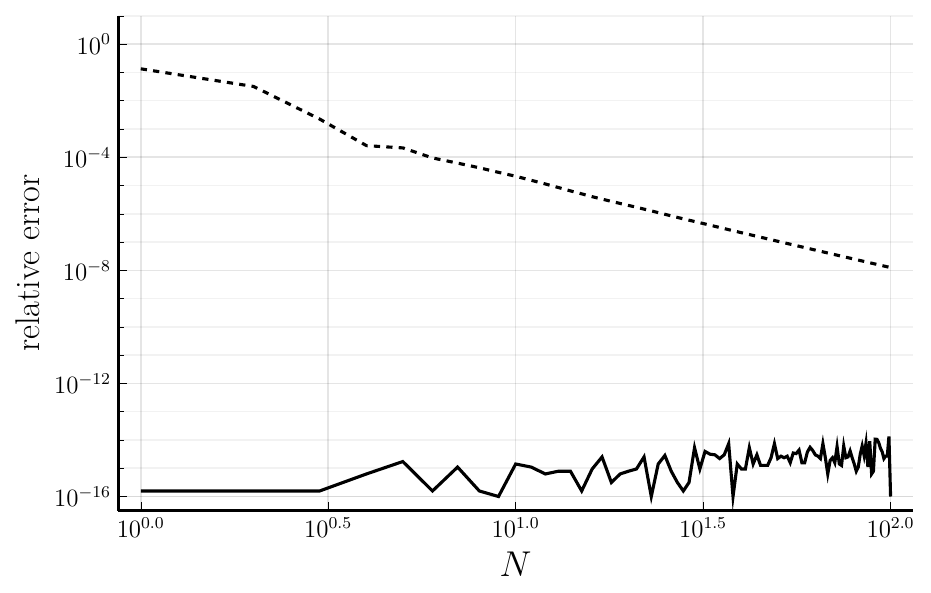}
	\caption{\label{fig:super-conv} Convergence of the $\Htwo$-norm
		of~\eqref{eq:rdde} with $A_0 = A_1 = -1$, $\tau = 1$, and
		$B = C = 1$, for different discretizations. The solid line
		corresponds to the Lanczos tau method with Chebyshev polynomials
		of the second kind, the dashed line to pseudospectral
		collocation in Chebyshev extremal nodes. For clarity of the
		figure, the relative error was lower bounded by $10^{-16}$.}
\end{figure}

\begin{figure}
	\centering
	\includegraphics[width=.7\textwidth]{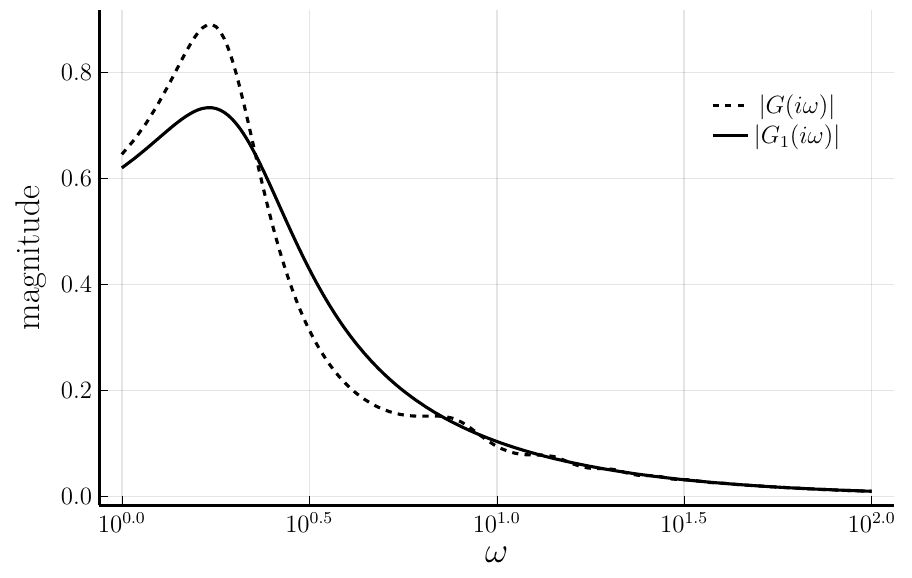}
	\caption{\label{fig:tf-approx} Transfer function $G$ of the system
		of \cref{fig:super-conv}, and the approximation $G_1$ from a
		Lanczos tau method using Chebyshev polynomials of the second
		kind with $N=1$.}
\end{figure}

For the class of systems showing this super convergence, we get as
transfer function
\begin{equation}\label{eq:degen-tf}
	G(s) = \pqty{s - a \pqty{1 + e^{-\tau s}}}^{-1}.
\end{equation}
We can compute its $\Htwo$-norm analytically by solving the delay Lyapunov
equation. As before, we have $\norm{G}_\Htwo = \sqrt{\lambda(0)}$, where
\[
	\begin{cases}
		\lambda'(t) + a\pqty{\lambda(t) + \lambda(t + \tau)} = 0, \quad t \in [-\tau, 0), \\
		\lambda(-t) = \lambda(t),                                                         \\
		2a\pqty{\lambda(0) + \lambda(\tau)} = -1.
	\end{cases}
\]
It is easy to verify that $\lambda(t) = \frac{a\tau - 2a\abs{t} - 1}{4a}$
is the solution, yielding
\[
	\norm{G}_\Htwo = \sqrt{\frac{a\tau - 1}{4a}}.
\]

We can now show the following.
\begin{proposition}\label{prop:superconv}
	For the system with transfer function~\eqref{eq:degen-tf}, we have
	\[
		\norm{G_N}_\Htwo = \norm{G}_\Htwo
	\]
	for $N \ge 1$, when discretized using a Lanczos tau method
	satisfying \cref{assum:oddeven}.
\end{proposition}
\begin{proof}
	From \cref{prop:delaylyap} we have
	$\norm{G_N}_\Htwo = \sqrt{\opE_0 U \opE_0^T}$ with $U$ the solution of
	\[
		\begin{cases}
			\opD U \opE_0^T +  a \opT_{N-1} U \pqty{\opE_0^T + \opE_{-\tau}^T} = \vb{0}, \\
			\opD U \opT^T_{N-1} + \opT_{N-1} U \opD^T = \vb{0},                          \\
			\opE_\theta U \opE^T_{\theta'} = \pqty{\opE_{\theta'} U \opE^T_\theta}^T,    \\
			a\pqty{\opE_0 + \opE_{-\tau}} U \opE^T_0 + a \opE_0 U \pqty{\opE_0^T + \opE_{-\tau}^T} = -1.
		\end{cases}
	\]

	Let $\mu$ be the univariate polynomial $U \opE_0^T$. From
	\cref{lem:reversal}, we then have
	$U\opE_{-\tau}^T = \opR U\opE_0^T = \opR\mu$. For the top equation
	this gives
	\begin{equation}\label{eq:scprf1}
		\opD \mu = -a\opT_{N - 1} \pqty{\mu + \opR\mu}.
	\end{equation}
	The bottom equation can be rewritten as
	$\pqty{\opE_0 + \opE_0\opR} \mu + \opE_0 \pqty{\mu + \opR\mu} =
		-\frac{1}{a}$ or
	\begin{equation}\label{eq:scprf2}
		\opE_0 \pqty{\mu + \opR\mu} = -\frac{1}{2a}.
	\end{equation}

	By expressing these equations in a basis, it is easily seen
	that~\eqref{eq:scprf1} and~\eqref{eq:scprf2} form a system of full
	rank, hence this system uniquely determines a polynomial. It is
	straightforwardly verified that
	$\mu(\theta) = \lambda(-\theta) = \frac{a\tau + 2a\theta - 1}{4a}$
	solves this system for a degree-graded, orthogonal basis, if
	$N \ge 1$, as $\mu + \opR\mu$ is a constant function. We thus get
	the exact result under the given conditions.
\end{proof}

\section{Conclusions}
We developed a framework of operator pencil formulations of the
Lanczos tau method~\eqref{eq:tauop} for the discretization of linear
systems with state delay.  The interpretation in terms of actions on
polynomials aids theoretical derivations.  We showed equivalence to
rational approximation in the frequency domain and provided an
explicit expression of the resulting rational function in
\cref{prop:taurat,prop:ratexp}, respectively.

Links were also made to pseudospectral collocation in \cref{thm:eq}.
We illustrated how Lanczos tau methods naturally lead to nested and
sparse matrix realizations, which can be exploited in Krylov methods,
allowing improved performance (\cref{sec:selfnest}).  Particularly
surprising was equivalence to Padé approximation for the choice of a
shifted Legendre basis (\cref{thm:pade}); our proof of which strongly
relied on the interpretation in terms of operations on polynomials.

Finally, we illustrated the potential benefits of the Lanczos tau
framework in \cref{sec:H2ex}, where, under a mild symmetry condition
(\cref{assum:oddeven}), significantly increased convergence rates,
compared to earlier work, were observed and partially proved for the
$\Htwo$-norm.  From links to the delay Lyapunov equation
(\cref{prop:delaylyap}), through bivariate polynomials, and from
qualitative properties of the rational approximation
(\cref{prop:ratmag}), we could provide intuitions for observed
super-geometric convergence.  Proving this effect, however, remains an
open problem.  A proof of a case of super convergence concluded our
work (\cref{prop:superconv}).  Note that this super-geometric
convergence and, in particular, super convergence are at first glance
somewhat unexpected, as the $\Htwo$-norm is inherently a global
characteristic; it depends on the behaviour of the transfer function
along the entire imaginary axis.

Although we did not pursue multiple nor distributed delays in this paper, the
development of \cref{sec:tau}, in particular \cref{prop:taurat,prop:ratexp}, the
equivalence \cref{thm:eq}, and the benefits discussed in \cref{sec:selfnest},
can easily be extended to these cases. In fact, any bounded linear functional
can be substituted for $A_0 \opE_0 + A_1 \opE_{-\tau}$ in~\eqref{eq:tauop}. As
already noted at the end of \cref{sec:pade}, a link to Padé approximation cannot
be obtained from such an extension. \Cref{sec:H2ex} also depends on having only
a single delay; how to extend the noted benefits to more general settings is
thus an interesting open question.

\appendix
\sloppy 
\printbibliography

\end{document}